\documentclass[12pt]{article}
\usepackage{authblk}
\usepackage{amsthm}
\usepackage[numbers]{natbib}
\usepackage[colorlinks,citecolor=blue]{hyperref}
\usepackage{amsmath}
\allowdisplaybreaks
\usepackage{graphicx, color, bm, cases} 
\newtheorem{thm}{Theorem}[section]

\newtheorem{definition}[thm]{Definition}

\newtheorem*{hypo*}{Hypothesis}

\newtheorem*{problem*}{Problem}
\newtheorem{lemma}[thm]{Lemma}

\newtheorem{prop}[thm]{Proposition}
\newtheorem{rem}[thm]{Remark}

\renewcommand {\theequation}{\thesection.\arabic{equation}}

\def\ep{{\epsilon}}
\def\la{{\lambda}}\def\si{{\sigma}}

\def\Om{{\Omega}}
\def\th{{\theta}}\def\Th{{\Theta}}

\def\<{\left<}\def\>{\right>}\def\({\left(}\def\){\right)}

\newfam\msbmfam\font\tenmsbm=msbm10\textfont
\msbmfam=\tenmsbm\font\sevenmsbm=msbm7

\scriptfont\msbmfam=\sevenmsbm\def\bb#1{{\fam\msbmfam #1}}

\def\EE{\bb E}\def\FF{\bb F}\def\GG{\bb G}\def\HH{\bb H}

\def\RR{\bb R}

\def\cF{{\cal F}}\def\cG{{\cal G}}

\def\cM{{\cal M}}
\def\cP{{\cal P}}\def\cU{{\cal U}}

\title{\large \bf  Stochastic filtering under model ambiguity}
		
		\author[1,2]{Jiaqi Zhang\thanks{Email:{\tt 11849459@mail.sustech.edu.cn}.}}
		\author[2,3]{Jie Xiong\thanks{Supported by National Key R\&D Program of China grant (No.2022YFA1006102) and NSFC grant (No.11831010). Email: {\tt xiongj@sustech.edu.cn}.}}
		\date{}
		
		\affil[1]{\footnotesize Department of Mathematics, Harbin Institute of Technology, Harbin, 150001, China}
		\affil[2]{\footnotesize Department of Mathematics, Southern University of Science and Technology, Shenzhen, Guangdong, 518055, China}
		\affil[3]{\footnotesize SUSTech International Center for Mathematics, Southern University of Science and Technology, Shenzhen, Guangdong, 518055, China}

\begin{document}
			\setlength{\baselineskip}{12pt}
			\maketitle

		\begin{abstract}
			In this paper, we study a non-linear filtering problem in the presence of signal model uncertainty. The model ambiguity is characterized by a class of probability measures from which the true one is taken. After interchanging the order of extremum problems by using the mini-max theorem, we find that the uncertain filtering problem can be converted to a weighted conditional mean-field optimal control problem.
			Further, we characterize the ambiguity filter and prove its unique existence.
		\end{abstract}

		

	\section{Introduction}
	\setcounter{equation}{0}
	\renewcommand{\theequation}{\thesection.\arabic{equation}}
	
	Originally motivated by its application in telecommunications, stochastic filtering has been studied extensively since the early work of \citet{Stratonovich1960,Stratonovic1966} and \citet{Kushner1964,Kushner1967}.
	The celebrated paper \citet{Fujisaki1972} marked the pinnacle of the innovative approach to non-linear filtering of diffusion processes. The optimal filtering equation is a non-linear stochastic partial differential equation (SPDE), which is usually called the Kushner–Stratonovich equation or the Kushner–FKK equation. The groundbreaking contributions of \citet{kallianpur1968,kallianpur1969} established the representation
	of the optimal filter in terms of the unnormalized one, which
	was studied in the pioneering doctoral dissertations of \citet{duncan1967},
	\citet{mortensen1966} and the important paper of \citet{zakai1969}. The linear SPDE governing the dynamics of the unnormalized filter is commonly referred to as the Duncan-Mortensen-Zakai equation, or more succinctly, Zakai's equation. For a more comprehensive and detailed introduction to nonlinear filtering, we refer the reader to the books of \citet{BC}, \citet{K:c}, \citet{LS1,LS2},  and \citet{X}.
	
	Recently, stochastic filtering has gained significant relevance in mathematical finance due to its diverse applications. In this context, observation processes commonly involve the prices of stocks or other securities, along with their derivatives. The associated quantities, such as the appreciation rates, serve as the essential ``signal'' that requires precise estimation through stochastic filtering methods. We refer the reader to the papers of \citet{BX}, \citet{HWW}, \citet{Lak}, \citet{NP}, \citet{Rogers}, \citet{Xia2}, \citet{xiong2007}, \citet{xiong2020}, and \citet{Zen}, for some examples. A related topic worth mentioning is the so-called ``optimal control under partial information'', which has captured the attention of numerous researchers. Here we mention a few works of \citet{os}, \citet{HWX}, and \citet{wang2013, wang2015}. We refer to the book of \citet{wang2018} for a detailed introduction to this topic.
	
	The fundamental premise of classical stochastic filtering rests on the perfect modeling of both the signal and observation processes. Nonetheless, this assumption may not always be tenable in various practical applications. Particularly, model ambiguity frequently arises in mathematical finance, as evident in works such as \citet{chen2002}, \citet{chen2012} and \citet{EJ}. 
	Thus, the objective of this article is to delve into the filtering problem under model ambiguity.
	
	Consider the following filtering model with real-valued signal and observation processes:
	
	\begin{equation}\label{intro-filter}
		\left\{\begin{array}{ccl}
			dX_t&=&b(X_t)dt+\si(X_t)dW_t,\quad X_0=x,\\
			dY_t&=&h(X_t)dt+dB_t,\quad Y_0=0,
		\end{array}\right.
	\end{equation}
	where the coefficients $b,\si$ and $h$ are continuous real functions, $(W_t,B_t)$ is a 2-dimensional standard Brownian motion under probability measure $P\in\cP$. In this context, the probability measure $P$ serves as a precise evaluation criterion for the signal process observed by external observers. The probability measure set $\cP$ is considered to encompass all evaluation criteria with the ambiguity parameters $\th\in\Th$. Then we naturally seek to estimate the signal process by minimizing the squared error in the worst-case scenario: 
	
	\[
	\min_{\xi}\max_{P\in\cP} \EE^P[|X_t-\xi|^2],
	\]
	where $\EE^P$ is the expectation with respect to the probability measure $P$, and $\xi$ is over all $\cG$-measurable random variables. Here $\cG\equiv\si(Y_s:s\le t)$.
	
	In this paper, we first prove the unique existence of the optimal control that minimizes the squared error within the most unfavorable evaluation criteria, relying on partially observable information under some necessary mathematical assumptions. Afterwards, by utilizing the mini-max theorem, we interchange the order of extremum problems and characterize the optimal control. Furthermore, we obtain the most favorable evaluation criteria, namely, the optimal probability measure.
	
	The rest of this article is organized as follows. In Section \ref{sec2}, we introduce the stochastic filtering problem under model ambiguity and state the main results of this article. The unique existence of the ambiguity filter is proved in Section \ref{sec3}. Section \ref{sec4} is devoted to the characterization of the ambiguity filter and the optimal probability measure by converting the filtering problem to a weighted conditional mean-field optimal control problem.  
	
	\section{Problem formulation and main results}\label{sec2}
	\setcounter{equation}{0}
	\renewcommand{\theequation}{\thesection.\arabic{equation}}
	
	Let $T>0$ be a fixed time horizon. Let $(\Omega,\cF,\FF\equiv\{\cF_t\}_{0\le t \le T},\bb P)$ be a complete filtered probability space satisfying the usual condition, on which two standard independent $\cF_t$-Brownian motions $W_t$ and $B_t$ are defined, where $\cF_t=\cF_t^{W,B}$ is their natural filtration and $\cF=\cF_T$. Let $\RR^n$ be the $n$-dimensional real Euclidean space and $|\cdot|$ be the norm in a Euclidean space. We denote by $C_b(\RR^d)$ the set of all bounded and continuous mappings on $\RR^d$, $L^p(\Om,\cF,P)$ the space of all the $\cF$-measurable $p$-power integrable random variables $\xi$ with $\|\xi\|_p\equiv (\EE[\xi^p])^{1/p}$ and, $L_\FF^p(0,T;\RR^d)$ the set of all $\RR^d$-valued $\mathcal{F}_t$-adapted processes $\phi_t$ such that for $p>1$,
	
	\[
	\EE\left[\int_0^T|\phi_s|^pds\right]<\infty.
	\] 
	Throughout this paper, all processes will be $\mathcal{F}_t$-adapted unless stated otherwise.
	
	The signal process $X_t$, or a function $f(X_t)$ of it, is what we want to estimate and the observation process $Y_t$ provides the information we can use. 
	Namely, if the model is without ambiguity, we look for a $\cG_t\equiv\si(Y_s:\;s\le t)$-adapted process $u_t$ such that for any $t\in[0,T]$, $\EE\left[|f(X_t)-u_t|^2\right]$ is minimized. It is clear that this $u_t$ also minimizes the quantity
	$\EE\left[\int^T_0|f(X_t)-u_t)|^2dt\right]$. On the other hand, if the minimizer of the latter quantity is unique which we will prove in this article, then it must
	coincide with the unique minimizer of the former. Thus, these two minimization problems are essentially equivalent. However, the latter is more convenient from a control point of view.
	\begin{definition}
		A control process $u_t$ is called admissible if it belongs to $L^2_\GG(0,T;\RR)$.
		The set of all admissible controls is denoted by $\cU_{ad}$.
	\end{definition}
	Now we are ready to introduce the stochastic filtering problem with drift ambiguity. For every $\th_t$ belonging to 
	
	\begin{equation}\label{def-th}
		\Th=\{(\th_t)_{t\in[0,T]}:\th_t\in\cF_t,\mbox{ and }\sup_{t\in[0,T]}|\th_t|\le k\},
	\end{equation}
	let $\cP$ be a class of probability measures which is defined as
	
	\begin{equation}\label{def-cP}
		\cP=\left\{Q\sim P: \frac{dQ}{dP}=\exp\left(\int^T_0\th_sdW_s-\frac12\int^T_0\th_s^2ds\right),\;\mbox{ with }\th\in\Th\right\},
	\end{equation}
	where $k$ is a non-negative constant, and $\th_t\in\cF_t$ means that $\th_t$ is $\cF_t$-measurable. 
	The cost functional associated with the control $u$ and the probability measure $Q\in\cP$ is defined as
	
	\begin{equation}\label{def-j}
		J(u,Q)=\EE^Q\left[\int_0^T|f(X_t)-u_t|^2dt\right],
	\end{equation}
	where $f\in C_b(\RR)$.
	The model ambiguity means that the true probability measure is one taken from $\cP$. This is also equivalent to drift ambiguity because by Girsanov's formula, $\widetilde{W}_t\equiv W_t-\int^t_0\th_sds$ is a Brownian motion and, under $Q$, $X_t$ is a diffusion process with drift coefficient $b+\si\th$. Namely, the signal process $X_t$ can be rewritten as 
	
	\begin{equation}\label{eq1102a}
		dX_t=(b(X_t)+\si(X_t)\th_t)dt+\si(X_t)d\widetilde{W}_t,\qquad X_0=x.
	\end{equation} 
	To simplify the notation,  we assume that $\si(x)\ge 0$ for all $x\in\RR$. 
	
	Before we proceed further, we would like to point out that we can also consider the ambiguity of the observation model by modifying \eqref{def-cP} by changing the formula there for $\frac{dQ}{dP}$ to the following
	
	\[\frac{dQ}{dP}=\exp\(\int^T_0\(\th_sdW_s+\tilde{\th}_sdB_s\)-\frac12\int^T_0\(\th^2_s+\tilde{\th}^2_s\)ds\)\]
	with $|\th_s|\le k$ and $|\tilde{\th}_s|\le\tilde{k}$, where $\tilde{k}$ is another non-negative constant. We choose to take $\tilde{k}=0$ for simplicity of notation since the arguments are similar.
	
	Throughout this paper, we impose the following hypotheses.
	\begin{hypo*}[H1]
		The functions $b,\sigma,f,h$ are continuously differentiable with respect to $x$ and their partial derivatives $b_x,\si_x,f_x,h_x$ are uniformly bounded. Further, $f,h$ are bounded functions.
	\end{hypo*}
	Because of the model ambiguity, we naturally consider the square error in the worst case scenario. 
	\begin{problem*}[O]
		For a given initial state $x\in\RR$, under Hypothesis (H1), seek a control $u\in\cU_{ad}$ such that 
		
		\[J(u)=\inf_{v\in\cU_{ad}}\sup_{Q\in\cP}J(v,Q)\equiv\inf_{v\in\cU_{ad}}J(v)\equiv J_0,\]
		subject to \eqref{def-j}. If such an identity holds, we call $u$ the ambiguity filter of $f(X_t)$.
	\end{problem*}
	\begin{rem}
		Note that $u_t$ in the definition above is $f$-dependent. We omit this dependence for simplicity. We also point out that the ambiguity filter coincides with the classical one if $\cP=\{P\}$ contains a single probability measure only.
	\end{rem}
	Next, we present the main results of this paper. The rigorous proofs are deferred to the subsequent sections. First, we establish the existence and uniqueness of the ambiguity filter. 
	\begin{thm}\label{thm1.1}
		Suppose that Hypothesis (H1) holds. For each initial state $x\in\RR$, Problem (O) admits a unique ambiguity filter.
	\end{thm}
	We proceed to characterizing the ambiguity filter which is the second main result of this article. 
	For each $Q\in\cP$, we define another probability measure $\widetilde{Q}$ such that $\widetilde{Q}\sim Q$ with Radon-Nikodym derivative given by
	
	\begin{equation}\label{eq0905b}
		\left.\frac{d\widetilde{Q}}{dQ} \right|_{\mathcal{F}_t}=M_t^{-1}\equiv\exp\(-\int^t_0h(X_s)dB_s-\frac12\int^t_0h(X_s)^2ds\).\end{equation}
	By Hypothesis (H1), due to the boundedness of $h$, the Novikov's condition holds.  Note that, under the probability measure $\widetilde{Q}$,  $Y_t$ is a Brownian motion independent of $\widetilde{W}_t$, and
	
	\begin{equation}\label{eq1102b}
		dM_t=h(X_t)M_tdY_t,\qquad M_0=1.
	\end{equation}
	The adjoint processes $(p_t,q_t,P_t,Q_t)$ are governed by the following backward stochastic differential equations (BSDEs):
	
	\begin{equation}\label{eq0905a}
		\left\{
		\begin{aligned}
			dp_t
			&=-\left\{h(X_t)q_t+\frac12\left(f(X_t)-\frac{\tilde{\EE}[f(X_t)M_t|\cG_t]}{\tilde{\EE}[M_t|\cG_t]}\right)^2\right\}dt+q_tdY_t,\\[0.5em]
			dP_t
			&=-\bigg\{(b'(X_t)+\si'(X_t)\th_t)P_t+\si'(X_t)Q_t-h'(X_t)M_t\big(q_t+h(X_t)p_t\big)\\
			&\quad\,+f'(X_t)M_t\biggl(f(X_t)-\frac{\tilde{\EE}[f(X_t)M_t|\cG_t]}{\tilde{\EE}[M_t|\cG_t]}\biggr)\bigg\}dt+Q_td\widetilde{W}_t,\\
			p_T&=0,\quad P_T=0,
		\end{aligned}
		\right.
	\end{equation}
	where $\tilde{\EE}$ denotes the expectation with respect to probability measure $\tilde{Q}$.
	
	\begin{thm}\label{thm1.2}
		Let Hypothesis (H1) hold. For each $\th\in\Th$ in (\ref{def-th}) being fixed, the forward and backward stochastic differential equation (FBSDE) (\ref{eq1102a}, \ref{eq1102b}, \ref{eq0905a}) has a unique solution. Further, the optimal ambiguity filter satisfying the SDE \eqref{fkk} is given by
		
		\begin{equation}\label{eq-u-frac}
			u_t=\frac{\tilde{\EE}[f(X_t)M_t|\cG_t]}{\tilde{\EE}[M_t|\cG_t]},
		\end{equation}
		where $\widetilde{Q}$ is defined through (\ref{def-cP}) and (\ref{eq0905b}) with $\th_t=k \mbox{ sgn}(P_t)$.
	\end{thm}

	\section{Existence and uniqueness of the ambiguity filter}\label{sec3}
	\setcounter{equation}{0}
	\renewcommand{\theequation}{\thesection.\arabic{equation}}
	
	In this section, we proceed to prove Theorem \ref{thm1.1}. 
	Denote by $V^Q_t$ as the conditional expectation of the total square error in the time interval $[t,T]$ with respect to an admissible measure $Q$:
	
	\begin{equation}\label{eq1210a}
		V^Q_t=\EE^Q\bigg[\int^T_t|f(X_s)-u_s|^2ds\;\Big|\;\mathcal{F}_t\bigg],\quad 0\le t\le T.\end{equation}
	Let
	
	\begin{equation}\label{eq1210b}
		y_t=\sup_{Q\in\cP}V^Q_t,\quad \mbox{and}\quad J(u)=\sup_{Q\in\cP}J(u,Q),\end{equation}
	where $J(u,Q)$ is defined in \eqref{def-j}. It is easy to see that $J(u)=y_0$. 
	\begin{thm}\label{thm1210a}
		The process $y_t$ is the unique solution to the BSDE
		
		\begin{equation}\label{eq:y_t}
			\left\{\begin{array}{lll}
				dy_t&=&\big(-|f(X_t)-u_t|^2+k|z_t|\big)dt+z_tdW_t+\tilde{z}_tdB_t, \\[0.5em]
				y_T&=&0.
			\end{array}\right.
		\end{equation}
	\end{thm}
	\begin{proof}
		It is clear that  $V^Q_t+\int^t_0|f(X_s)-u_s|^2ds$ is a martingale under the probability measure $Q$. The martingale representation theorem implies that $V^Q_t$ is a solution to the BSDE:
		
		\[
		\left\{\begin{array}{lll}
			dV^Q_t&=& z_tdW_t+\tilde{z}_tdB_t-|f(X_t)-u_t|^2dt-z_t\th_tdt,\\[0.5em]
			V^Q_T&=&0,
		\end{array}\right.\]
		where $z_t$ and $\tilde{z}_t$ are predictable processes with respect to the filtration $(\mathcal{F}_t)$. Note that the probability uncertainty only reflects on the drift, and hence, $z_t$ and $\tilde{z}_t$ do not depend on the probability measure $Q$.
		
		Note that
		
		\[\max_{\th\in\Th}\th_tz_t=k|z_t|.\]
		Then,
		
		\begin{eqnarray*}
			y_t &=& \sup_{Q\in\cP}V^Q_t\\
			&=&\sup_{Q\in\cP}\bigg\{\int_t^T \big(|f(X_s)-u_s|^2+z_s\th_s\big)ds-\int_t^T z_sdW_s-\int_t^T \tilde{z}_sdB_s\bigg\} \\
			&\le&\sup_{Q\in\cP}\bigg\{\int_t^T \big(|f(X_s)-u_s|^2+k|z_s|\big)ds-\int_t^T z_sdW_s-\int_t^T \tilde{z}_sdB_s\bigg\}\\
			&=&\int_t^T \Big(|f(X_s)-u_s|^2+k|z_s|\Big)ds-\int_t^T z_sdW_s-\int_0^t \tilde{z}_sdB_s.
		\end{eqnarray*}
		
		On the other hand, by Lemma B.1(b) in \citet{chen2002}, there exists $\th^*_t\in\Th$  such that
		
		\[
		\th^*_tz_t=\max_{\th_t}\th_tz_t=k|z_t|.
		\]
		Hence,
		
		\begin{eqnarray*}
			y_t &=& \sup_{Q\in\cP}V^Q_t\\
			&=&\sup_{Q\in\cP}\bigg\{\int_t^T \big(|f(X_s)-u_s|^2+z_s\th_s\big)ds-\int_t^T z_sdW_s-\int_t^T \tilde{z}_sdB_s\bigg\} \\
			&\ge& \int_t^T \Big(|f(X_s)-u_s|^2+\th^*_tz_t\Big)ds-\int_t^T z_sdW_s-\int_t^T \tilde{z}_sdB_s\\
			&=&\int_t^T \Big(|f(X_s)-u_s|^2+k|z_t|\Big)ds-\int_t^T z_sdW_s-\int_t^T \tilde{z}_sdB_s,
		\end{eqnarray*}
		which implies that $y_t$ is a solution to (\ref{eq:y_t}). The uniqueness follows from the standard result of BSDE since the coefficients satisfy Lipschitz's continuity. This finishes the proof.
	\end{proof}
	
	Note that $y_t$ can also be represented as the unique solution to the following BSDE
	
	\begin{equation}\label{eq1210c}
		\left\{\begin{array}{ccl}
			dy_t&=&\big(-|f(X_t)-u_t|^2-h(X_t)\tilde{z}_t+k|z_t|\big)dt+z_tdW_t+\tilde{z}_tdY_t,\\[0.5em]
			y_T&=&0.
		\end{array}\right.\end{equation}
	
	Before proceeding with the proof of Theorem \ref{thm1.1}, it is essential to lay the groundwork with the following preparation.
	
	\begin{lemma}\label{lem1011a}
		To search for optimal control, we can restrict the admissible one to those $u$ with $\|u\|_\infty\le\|f\|_\infty$, 
		where $\|\cdot\|_\infty$ denotes the supreme norm.
	\end{lemma}
	\begin{proof}
		For any $u\in\cU_{ad}$, we define
		
		\[\tilde{u}_t=\left\{\begin{array}{cl}
			u_t,&\quad\mbox{ if } |u_t|\le\|f\|_\infty,\\[0.5em]
			\|f\|_\infty,&\quad\mbox{ if } u_t>\|f\|_\infty,\\[0.5em]
			-\|f\|_\infty,&\quad\mbox{ if } u_t<-\|f\|_\infty.
		\end{array}\right.\]
		It is easy to show that
		
		\[|f(X_t)-\tilde{u}_t|\le |f(X_t)-u_t|,\]
		and hence, $J(\tilde{u},Q)\le J(u,Q)$. This implies that $J(\tilde{u})\le J(u)$.
	\end{proof}
	
	\begin{proof}[Proof of Theorem \ref{thm1.1}]
		Let $u^n\in\cU_{ad}$ be such that $J(u^n)\to J_0$. By Lemma \ref{lem1011a}, without loss of generality, we may and will assume that $\|u^n\|_\infty\le\|f\|_\infty$. Then, $\{u^n\}$ is bounded in $\HH\equiv L^2([0,T]\times\Om)$ and hence, it is compact in the weak topology of $\HH$. Without loss of generality, we assume that
		$u^n\to u$ in the weak topology. By Mazur's theorem, there is a sequence of convex combinations
		
		\[\hat{u}^n=\sum_j\la^n_j u^{n+j}\to u\]
		in the strong topology of $\HH$, where $\la^n_j\ge 0$ with $\sum_j\la^n_j =1$. By the convexity of $\cU_{ad}$,  $\hat{u}^n\in\cU_{ad}$ and hence the limit $u\in\cU_{ad}$.
		
		Let $\hat{y}^n_t$ be given through (\ref{eq1210a}) and (\ref{eq1210b}) with $u$ being replaced by $\hat{u}^n$.
		Similar to Theorem \ref{thm1210a}, $\hat{y}^n_t$, together with $(\hat{z}^n_t,\tilde{\hat{z}}^n_t)$ is the unique solution to BSDE
		
		\begin{equation}\label{eq-hat-yn}
			\left\{\begin{array}{ccl}
				d\hat{y}^n_t&=&\(k|\hat{z}^n_t|-|f(X_t)-\hat{u}^n_t|^2-h(X_t)\hat{\tilde{z}}^n_t\)dt+\hat{z}^n_tdW_t+\hat{\tilde{z}}^n_tdY_t,\\
				\hat{y}^n_T&=&0.\end{array}\right.
		\end{equation}
		
		Note that
		
		\[|f(X_t)-\hat{u}^n_t|^2\le\sum_j\la^n_j|f(X_t)-u^{n+j}_t|^2.\]
		Then,
		
		\begin{eqnarray}\label{eq1210d}
			\hat{y}^n_t
			&=&\sup_{Q\in\cP}\EE^Q\left[\int^T_t|f(X_s)-\hat{u}^n_s|^2ds\;\Big|\;\mathcal{F}_t\right]\nonumber\\
			&\le&\sup_{Q\in\cP}\sum_j\la^n_j\EE^Q\left[\int^T_t|f(X_s)-u^{n+j}_s|^2ds\;\Big|\;\mathcal{F}_t\right]\nonumber\\
			&\le&\sum_j\la^n_j\sup_{Q\in\cP}\EE^Q\left[\int^T_t|f(X_s)-u^{n+j}_s|^2ds\;\Big|\;\mathcal{F}_t\right]\nonumber\\
			&=&\sum_j\la^n_jy^{n+j}_t.
		\end{eqnarray}
		Thus,
		
		\[J(\hat{u}^n)=\hat{y}^n_0\le \sum_j\la^n_j y^{n+j}_0= \sum_j\la^n_j J(u^{n+j}).\]
		For any $\ep>0$, let $N>0$ be such that $J(u^n)<J_0+\ep$ for all $n\ge N$. Then,
		
		\begin{equation}\label{eq-j0}
			J_0\le J(\hat{u}^n)\le \sum_j\la^n_j (J_0+\ep)=J_0+\ep.\end{equation}
		According to  \eqref{eq1210c}, \eqref{eq-hat-yn} and Lemma \ref{lem1011a}, applying It\^{o}'s formula to $|\hat{y}^n_t-y_t|^2$, we derive that
		
		\begin{align}\label{ineq-y-ep}
			&\EE\left[|\hat{y}^n_t-y_t|^2\right]+\EE\left[\int_t^T\(|\hat{z}^n_s-z_s|^2+|\hat{\tilde{z}}^n_s-\tilde{z}_s|^2\) ds\right]\nonumber\\
			=&2\EE\bigg[\int_t^T\<-\hat{y}^n_s+y_s,k(|\hat{z}^n_s|-|z_s|)-h(X_s)(\hat{\tilde{z}}^n_s-\tilde{z}_s\>ds\nonumber\\
			&\quad -\int_t^T\<-\hat{y}^n_s+y_s,\(|f(X_s)-\hat{u}^n_s|^2-|f(X_s)-u_s|^2\)\>ds\bigg]\nonumber\\
			\le& C\EE\left[ \int_t^T\(|\hat{y}^n_s-y_s|^2+|\hat{u}^n_s-u_s|^2\)ds\right]\nonumber\\
			&+\frac12\EE\left[\int_t^T\(|\hat{z}^n_s-z_s|^2+|\hat{\tilde{z}}^n_s-\tilde{z}_s|^2\) ds\right],
		\end{align}
		where $C>0$ is a constant. It follows from Gronwall's inequality that
		
		\[
		\EE\left[|\hat{y}^n_t-y_t|^2\right]\le e^{CT}\EE\left[\int_0^T |\hat{u}^n_t-u_t|^2dt\right], 
		\]
		which yields that 
		$J(\hat{u}^n)=\hat{y}^n_0\to y_0=J(u)$. By \eqref{eq-j0}, we get
		$J(u)=J_0$ and hence,  $u_t$ is an optimal ambiguity filter.
		
		The uniqueness follows from the convexity directly, while the convexity is obtained by comparison similar to (\ref{eq1210d}). The proof completes.
	\end{proof}

	\section{Characterization of the ambiguity filter}\label{sec4}
	\setcounter{equation}{0}
	\renewcommand{\theequation}{\thesection.\arabic{equation}}
	
	In this section, we use a weighted conditional mean-field approach to establish a necessary condition for the ambiguity filter. Namely, we proceed to present the proof of Theorem \ref{thm1.2}.
	
	\begin{lemma}\label{lem-cP-convex}
		The set of probability measures $\cP$ defined in \eqref{def-cP} is convex, and for any $p> 1$, the set $\{\frac{dQ}{dP}:Q\in\cP\}\subset L^p(\Om,\FF;P)$ is compact in the weak topology $\si(L^p(\Om,\FF;P),L^{1+\frac{p}{p-1}}(\Om,\FF;P))$.
	\end{lemma}
	\begin{proof} 
		The convexity of $\cP$ has been proved in \citet[Theorem 2.1]{chen2002}. Because of the boundedness of $\th$, by \citet[Lemma 4.1]{tang2023stochastic}, the set $\{\frac{dQ}{dP}:Q\in\cP\}$ is uniformly bounded in the norm $\|\cdot\|_{p}$. Then it follows from Theorem 4.1 of Chapter 1 in \citet{simons2008hahn} that the set $\{\frac{dQ}{dP}:Q\in\cP\}$
		is $\si(L^p(\Om,\FF;P),L^{1+\frac{p}{p-1}}(\Om,\FF;P))$-compact. This completes the proof.
	\end{proof}
	
	The convexity of $\cG$ and Lemma \ref{lem-cP-convex} allow us to apply the mini-max theorem (see Theorem B.1.2 in \citet{pham2009}) to the ambiguity filtering problem which can obtain the following theorem immediately.
	\begin{thm}
		Let Hypothesis (H1) hold. Then,
		
		\begin{equation}\label{eq-minimax}
			\min_{v\in\cU_{ad}}\sup_{Q\in\cP}J(v,Q)=\sup_{Q\in\cP}\min_{v\in\cU_{ad}}J(v,Q).
		\end{equation}
	\end{thm}
	
	Recall that the probability measure $\widetilde{Q}$ defined in \eqref{eq0905b} is absolutely continuous with respect to $Q$ and the Radon-Nikodym derivative $M_t^{-1}$ satisfies the following equation
	
	\begin{equation}\label{eq1102bb}
		dM_t=h(X_t)M_tdY_t,\quad M_0=1.
	\end{equation}
	
	We first fix $\th\in\Th$, and search for the optimal filter. 
	Under the probability measure $\tilde{Q}$ defined in \eqref{eq0905b}, $Y_t$ and $\widetilde{W}_t$
	are independent Brownian motions. Recall that the signal equation can  be rewritten as
	
	\[
	dX_t=(b(X_t)+\si(X_t)\th_t)dt+\si(X_t)d\widetilde{W}_t,\qquad X_0=x.
	\]
	Notice that $X_t$ is dependent with the parameter $\th_t$. To make the discussion clear, in what follows, we use the notation $X^\th_t$ to replace $X_t$. With $\th$ being fixed, we consider the control problem on the right side of \eqref{eq-minimax}. 
	
	\begin{problem*}[MC]
		With $\th$ being fixed in \eqref{def-cP} and the initial state $x\in \RR$ being given, we seek a control $u\in \cU_{ad}$ such that
		
		\[
		J(u,Q)=\inf_{v\in\cU_{ad}} J(v,Q),
		\]
		subject to \eqref{intro-filter} and \eqref{eq1102a}, where $J(\cdot,Q)$ is given by \eqref{def-j}. 
	\end{problem*}	  
	
	As we mentioned in Section \ref{sec2}, when fixed the parameter $\th$, which means the probability measure is fixed, Problem (MC) is equivalent to a classical optimal filtering problem.
	Applying filtering theory (we refer the reader to Chapter 5 in  \citet{X} for more details), 
	the optimal filter is a $\cG_t$-adapted 
	probability measure-valued process $\{\pi^\th_t(\cdot),t\in [0,T]\}$ given by
	
	\begin{equation}\label{eq-pi}
		\pi^\th_t(\phi)=\EE^Q[\phi(X^\th_t)|\cG_t] \mbox{ a.s., }
	\end{equation}
	for any $\phi\in C_b(\RR)$ and $t\in[0,T]$. The optimal control $u$ of Problem (MC) can be solved as
	
	\begin{equation}\label{eq1216a}
		u_t=\pi^\th_t(f).
	\end{equation}
	The innovation process $\nu_t$ defined by
	
	\begin{equation}\label{def-nu}
		\nu_t=Y_t-\int_0^t\pi^\th_s(h)\,ds,
	\end{equation}
	is a $\cG_t$-Brownian motion under probability measure $Q$. Note that the generator of the signal process 
	
	\begin{equation}\label{def-lf}
		L\phi(x)=\phi'(x)(b+\si\th)+\frac12\phi''(x)\si^2,\quad\forall t\in[0,T],\;\forall \phi\in C^2_b(\RR).
	\end{equation}
	
	The following Kushner-FKK equation for the optimal filter is taken from Theorem 5.7 in \citet{X}.
	
	\begin{prop}
		Let $\th$ be fixed in \eqref{def-cP}. Under Hypothesis (H1),  the optimal filter of Problem (MC) satisfies the following equation: for all $\phi\in C^2_b(\RR)$,
		
		\begin{equation}\label{fkk}
			\pi^\th_t(\phi)=\pi_0(\phi)+\int_0^t\pi^\th_s(L\phi)ds+\int_0^t\big(\pi^\th_s(h\phi)-(\pi^\th_sh)(\pi^\th_s\phi)\big)d\nu_s.
		\end{equation}
	\end{prop}
	
	Let $\cM_F(\RR)$ denote the space of all finite Borel measures on $\RR$. Define the $\cM_F(\RR)$-valued process $\{\rho^\th_t,t\in [0,T]\}$ on stochastic basis $(\Om,\cF,\tilde{Q},\cG_t)$ by 
	
	\begin{equation}\label{def-rho}
		\rho^\th_t(\phi)\equiv\tilde{\EE}[M_t\phi(X^\th_t)|\cG_t]\quad \forall t\in[0,T],\;\forall \phi\in C_b(\RR),
	\end{equation}
	where $M_t$ is defined in \eqref{eq1102bb} and $\tilde{\EE}$ is the expectation with respect to $\tilde{Q}$. $\rho^\th_t$ is known as the unnormalized filter. 
	Applying It\^{o}'s formula to $M_t\phi(X^\th_t)$ we can immediately arrive at the following Zakai equation.
	
	\begin{prop}
		Let $\th$ be fixed in \eqref{def-cP}. Under Hypothesis (H1), the unnormalized filter $\rho^\th_t$ satisfies the following equation: $\forall\phi \in C^2_b(\RR)$,
		
		\begin{equation}\label{zakai}
			\rho^\th_t(\phi)=\rho_0(\phi)+\int_0^t \rho^\th_s(L\phi)ds+\int_0^t\rho^\th_s(h\phi)dY_s.  
		\end{equation}
	\end{prop}
	
	\begin{rem}
		According to Theorem 2.21 in \citet{lucic2001}, for each $\th$ fixed in \eqref{def-cP}, namely, for each $Q\in\cP$, the normalized filter equation \eqref{fkk} has the property of uniqueness in law and the unnormalized filter equation \eqref{zakai} has the property of both pathwise uniqueness and uniqueness in law.
	\end{rem}
	
	In virtue of Kallianpur-Striebel formula, for fixed $\th\in\Th$, the optimal control $u$ of Problem (MC) given by \eqref{eq1216a} can also be represented as
	
	\begin{equation}\label{eq1216b}
		u_t=\pi^\th_t(f)=\frac{\rho^\th_t(f)}{\rho^\th_t(1)}.
	\end{equation}
	Plugging it into \eqref{eq-minimax}, Problem(O) then is converted from a mini-max problem into a weighed conditional mean-field optimal control problem with the control $\th\in\Th$, the cost functional 
	
	\begin{equation}\label{J-th}
		J(\th)=-\tilde{\EE}\left[\frac12\int^T_0\left|f(X^\th_t)-\frac{\tilde{\EE}[f(X_t^\th)M_t|\cG_t]}{\tilde{\EE}[M_t|\cG_t]}\right|^2M_tdt\right],
	\end{equation}
	and state process $(X^\th_t,M_t)$ satisfying (\ref{eq1102a}, \ref{eq1102bb}). Note that we have put the factor $-\frac12$ to switch the maximization problem to the minimization one. By \citet[Lemma 4.1]{tang2023stochastic}, for each $\th\in \Th$, the weighted state equations  (\ref{eq1102a}, \ref{eq1102bb}) admit a unique solution $(X_t^\th,M_t)\in L^2_{\FF}(0,T;\RR^{2})$.
	
	Suppose that $\th_t$ is the optimal control that minimizes the cost functional \eqref{J-th}, and $(X^\th_t,M_t)$ is the corresponding optimal state. Let $v_t$ be such that $\th_t+v_t\in\Th$. For any $\ep\in(0,1)$, by the convexity of $\Th$, we see that $\th_t+\ep v_t\in\Th$. We denote $(X^{\th+\ep v}_t,M^{\th+\ep v}_t)$ as the solution of (\ref{eq1102a}, \ref{eq1102bb}) along with the control $\th_t+\ep v_t$. We now present the the convergence of $(X^{\th+\ep v}_t,M^{\th+\ep v}_t)$ to $(X^\th_t,M_t)$ and establish the convergence rate. As the result can be readily obtained, we shall state it without including the proof.
	\begin{lemma}
		Let Hypothesis (H1) hold, then there exists a constant $K>0$ such that
		
		\[
		\tilde{\EE}\big[|X^{\th+\ep v}_t-X^\th_t|^2\big]+\Big(\tilde{\EE}\big[|M^{\th+\ep v}_t-M_t|\big]\Big)^2\le K\ep^2.
		\]
	\end{lemma}
	Define $(X^1_t,M^1_t)$ by the following variational equation: for any $v_t\in L^2_\FF(0,T;\RR)$,
	
	\begin{equation}\label{eq-x1m1}
		\left\{
		\begin{aligned}
			dX^1_t=&\,\bigl((b'(X^\th_t)+\si'(X^\th_t)\th_t)X^1_t+\si(X^\th_t)v_t\bigr)dt+\si'(X^\th_t)X^1_td\widetilde{W}_t,\\
			dM^1_t=&\,-h'(X^\th_t)h(X^\th_t)M_tX^1_tdt+\bigl(h(X^\th_t)M_t^1-h'(X^\th_t)M_tX_t^1\bigr)dY_t,\\
			X^1_0=&\,0,\quad M^1_0=0.
		\end{aligned}
		\right.
	\end{equation}
	For $v$ being fixed, under Hypothesis (H1), it follows from \citet[Proposition 2.1]{sun2014linear} that the variational equation \eqref{eq-x1m1} admits a unique pair of solutions $(X^1_t,M^1_t)\in L^2_\FF(0,T;\RR^{2})$. The following result can be estimated by a similar approach to \citet[Lemma 5.2]{tang2023stochastic}, which is stated without proof. 
	\begin{lemma}
		Let Hypothesis (H1) hold and 
		
		\[
		\chi^\ep_t=\ep^{-1}\bigl(\chi^{\th+\ep v}_t-\chi_t\bigr)-\chi^1_t,
		\]
		where $\chi=X,M$, then
		
		\[
		\lim_{\ep\to0}\tilde{\EE}\bigg[\int_0^T\Bigl(|X^\ep_t|^2+|M^\ep_t|^2\Bigr)dt\bigg]=0.
		\]
	\end{lemma}
	The next lemma is concerned with the perturbation of the cont functional defined in \eqref{J-th} with respect to the parameter $\ep$. For simplifying the notation, we define 
	
	\[
	l(t)=-\frac12\left|f(X^\th_t)-\frac{\tilde{\EE}[f(X_t^\th)M_t|\cG_t]}{\tilde{\EE}[M_t|\cG_t]}\right|^2M_t,
	\]
	and $l_x(t),l_m(t),l_{\rho_1}(t),l_{\rho_2}(t)$ as the corresponding partial derivation of $l$ with respect to $X^\th_t$, $M_t$, $\tilde{\EE}[f(X_t^\th)M_t|\cG_t]$, and $\tilde{\EE}[M_t|\cG_t]$, respectively, given by
	
	\begin{equation}\label{eq-lm}
		\left\{
		\begin{aligned}
			l_x(t)&=-f'(X^\th_t)M_t\biggl(f(X^\th_t)-\frac{\tilde{\EE}[f(X_t^\th)M_t|\cG_t]}{\tilde{\EE}[M_t|\cG_t]}\biggr),\\
			l_m(t)&=-\frac12\left(f(X^\th_t)-\frac{\tilde{\EE}[f(X_t^\th)M_t|\cG_t]}{\tilde{\EE}[M_t|\cG_t]}\right)^2,\\
			l_{\rho_1}(t)&=\frac{M_t}{\tilde{\EE}[M_t|\cG_t]}\biggl(f(X^\th_t)-\frac{\tilde{\EE}[f(X_t^\th)M_t|\cG_t]}{\tilde{\EE}[M_t|\cG_t]}\biggr),\\
			l_{\rho_2}(t)&=-M_t\biggl(f(X^\th_t)-\frac{\tilde{\EE}[f(X_t^\th)M_t|\cG_t]}{\tilde{\EE}[M_t|\cG_t]}\biggr)\frac{\tilde{\EE}[f(X_t^\th)M_t|\cG_t]}{\tilde{\EE}[M_t|\cG_t]^2}.
		\end{aligned}\right.
	\end{equation}
	By Hypothesis (H1) and Lemma \ref{lem-cP-convex}, we can derive that $l_x,l_m,l_{\rho_1},l_{\rho_2}\in L^p(0,T;\RR)$ for any $p\ge1$.
	\begin{lemma}
		Let Hypothesis (H1) hold, then
		
		\begin{align}\label{eq-j-ep=0}
			\frac{d}{d\ep}J(\th+\ep v)\Big|_{\ep=0}
			=\tilde{\EE}\bigg[\int_0^T\Big(l_x(t)X^1_t+l_m(t)M^1_t\Big)dt\bigg].
		\end{align}
	\end{lemma}
	Note that $\tilde{\EE}[l_{\rho_1}(t)|\cG_t]=\tilde{\EE}[l_{\rho_2}(t)|\cG_t]=0$. Plugging $X^{\th+\ep v}_t$, $M^{\th+\ep v}_t$, and $\th_t+\ep v_t$ into \eqref{J-th}, the result above can be obtained immediately after some derivative calculations, so we omit it.
	
	Recall the adjoint processes $(p_t,q_t,P_t,Q_t)$ are introduced in \eqref{eq0905a}. 
	In view of Hypothesis (H1) and \citet[Proposition 2.1]{sun2014linear}, once $X^\th_t$ and $\th_t$ are determined, the adjoint equation \eqref{eq0905a} admits a unique solution $(p_t,q_t,P_t,Q_t)\in L^2_\FF(0,T;\RR^{4})$. Now we are ready to estimate the optimal control $\th\in\Th$.
	\begin{thm}
		Let Hypothesis (H1) hold. Suppose $\th_t\in\Th$ is the optimal control that minimizes the cost functional defined in \eqref{J-th} and $X^\th_t$ is the corresponding optimal state. Then we have 
		
		\begin{equation*}
			\th_t=k\mbox{ sgn }(P_t).
		\end{equation*}
	\end{thm}
	\begin{proof}
		Combined with \eqref{eq-lm}, adjoint processes $(p_t,q_t,P_t,Q_t)$  can be rewritten as follows:
		
		\begin{equation}\label{eq-pP}
			\left\{
			\begin{aligned}
				dp_t
				&=\big(l_m-h(X^\th_t)q_t\big)dt+q_tdY_t,\\[0.5em]
				dP_t
				&=\Big\{l_x-\big(b'(X^\th_t)-\si'(X^\th_t)\th_t\big)P_t-\si'(X^\th_t)Q_t+h'(X^\th_t)M_tq_t\\
				&\quad+h'(X^\th_t)h(X^\th_t)p_t\Big\}dt+Q_td\widetilde{W}_t,\\
				p_T&=0,\quad P_T=0.
			\end{aligned}
			\right.
		\end{equation}
		Then by \eqref{eq-pP} and \eqref{eq-x1m1}, it follows from It\^{o}'s formula that
		
		\[\left\{\begin{aligned}
			dp_tM_t^1&=\Big(l_mM^1_t-(h(X^\th_t)p_t+q_t)h'(X^\th_t)M_tX^1_t\Big)dt\\
			&\quad\,+\Big(M^1_tq_t+p_t\big(h(X^\th_t)M^1_t-h'(X^\th_t)M_tX^1_t\big)\Big)dY_t,\\
			dP_tX^1_t&=\Big(l_xX^1_t+(h(X^\th_t)p_t+q_t)h'(X^\th_t)M_tX^1_t+\si(X^\th_t)P_tv_t\Big)dt\\
			&\quad\,+X^1_t\big(Q_t+\si'(X^\th_t)\big)X^1_td\widetilde{W}_t.
		\end{aligned}\right.
		\]
		Taking integral on both sides of the above SDEs, we can obtain that
		
		\begin{align}\label{eq-pm1px1}
			\tilde{\EE}\Big[p_TM_T^1+P_TX^1_T\Big]=\tilde{\EE}\bigg[\int_0^T\Big(l_mM^1_t+l_xX^1_t+\si(X^\th_t)P_tv_t\Big)dt\bigg].
		\end{align}
		Recall that $\th_t$ is an optimal control that minimizes the cost functional \eqref{J-th} in the sense that for all $v$ satisfying $\th+\ep v\in \Th$ with $\ep\in [ 0,1)$, $J(\th+\ep v)$ attains its minimum at $\ep=0$. Plugging \eqref{eq-pm1px1} back into \eqref{eq-j-ep=0}, since
		
		\[
		\lim_{\ep\to0^+}\ep^{-1}(J(\th+\ep v)-J(\th))\ge 0,
		\]
		we derive that
		
		\[
		\tilde{\EE}\bigg[\int_0^T\big(\si(X^\th_t)P_tv_t\big)dt\bigg]\le 0.
		\]
		Note that there exist $\th^0\in \Th$ such that $v_t=\th^0_t-\th_t$. Thus,
		
		\[
		\tilde{\EE}\bigg[\int_0^T\big(\si(X^\th_t)P_t(\th^0_t-\th_t)\big)dt\bigg]\le 0.
		\]
		Therefore, to ensure that the above inequality holds, in virtue of the assumption that $\si\ge 0$, we must have $\th_t=k\mbox{ sgn}(P_t)$.	This marks the conclusion of the current proof, while simultaneously accomplishing the proof of Theorem \ref{thm1.2}.
	\end{proof}

	\bibliographystyle{abbrvnat}
\bibliography{reference}

\end{document}